\documentclass[10pt,reqno]{amsart} 

\usepackage{amssymb,latexsym}
\usepackage{cite} 

\usepackage[height=190mm,width=130mm]{geometry} 

\theoremstyle{plain}
\newtheorem{theorem}{Theorem}
\newtheorem{lemma}{Lemma}

\newtheorem{proposition}{Proposition}

\theoremstyle{definition}
\newtheorem{definition}{Definition}

\theoremstyle{remark}
\newtheorem{remark}{Remark}

\newcommand{\Z}{\mathbb{Z}}
\newcommand{\C}{\mathbb{C}}
\newcommand{\R}{\mathbb{R}}

\numberwithin{equation}{section} 

\newcommand{\Ga}{\ensuremath{\Gamma}}

\newcommand{\rmap}{\longrightarrow}

\newcommand{\N}{\ensuremath{\mathbb {N}}}

\newcommand{\U}{\ensuremath{\mathcal{U}}}

\newcommand{\A}{\ensuremath{\mathcal{A}}}

\newcommand{\F}{\ensuremath{\mathcal{F}}}
\newcommand{\es}{\ensuremath{\mathcal{S}}}

\newcommand{\G}{\ensuremath{\mathcal{G}}}
\newcommand{\g}{\ensuremath{\Gamma}}

\newcommand{\ps}{{\raise 1pt\hbox{\tiny (}}}

\newcommand{\pss}{{\raise 1pt\hbox{\tiny [}}}
\newcommand{\pdd}{{\raise 1pt\hbox{\tiny ]}}}
\newcommand{\pd}{{\raise 1pt\hbox{\tiny )}}}

\newcommand{\bs}{{\raise 1pt\hbox{\tiny [}}}
\newcommand{\bd}{{\raise 1pt\hbox{\tiny ]}}}

\def\cross{\mathinner{\mathrel{\raise0.8pt\hbox{$\scriptstyle>$}}
                 \joinrel\mathrel\triangleleft}}

\usepackage{citehack}
\usepackage{amsmath,amsthm}
\usepackage{amsfonts}
\usepackage{amssymb}
\usepackage{longtable}
\usepackage[matrix,arrow,curve]{xy}

\usepackage{lipsum}

\def\U{\mathcal{U}}

\def\K{\mathcal{K}}

\usepackage{stackrel}

\newcommand{\be}{\begin{equation}}
\newcommand{\ee}{\end{equation}}

\newcommand{\nn}{\nonumber \\}

\newcommand{\wt}{\mbox{\rm wt}\ }

\newcommand{\nc}{\newcommand}
\nc{\cali}{\mathcal}
\nc{\on}{\operatorname}
\nc{\Wick}{{\mb :}}

\nc{\ddz}{\frac{\partial}{\partial z}}
\nc{\ch}{\mbox{ch}}
\nc{\Oo}{{\cali O}}
\nc{\cond}{|\,}
\nc{\bib}{\bibitem}
\nc{\pone}{\Pro^1}
\nc{\pa}{\partial}
\nc{\arr}{\rightarrow}
\nc{\larr}{\longrightarrow}
\nc{\ket}{)}
\nc{\bra}{(}
\nc{\gam}{\bar{\gamma}}
\nc{\q}{\widetilde{Q}}
\nc{\ep}{\epsilon}
\nc{\su}{\widehat{{\mf s}{\mf l}}_2}
\nc{\sw}{{\mf s}{\mf l}}
\nc{\h}{{\mf h}}
\nc{\n}{{\mf n}}
\nc{\ab}{\mf{a}}
\nc{\is}{{\mb i}}
\nc{\js}{{\mb j}}
\nc{\bi}{\bibitem}
\nc{\He}{{\cali H}}
\nc{\inv}{^{-1}}
\nc{\ol}{\overline}
\nc{\wh}{\widehat}
\nc{\dst}{\displaystyle}

\nc{\delt}{\partial_t}
\nc{\ddt}{\frac{\partial}{\partial t}}
\nc{\delx}{\partial_x}
\nc{\mb}{\mathbf}
\nc{\mf}{\mathfrak}

\nc{\mbb}{\mathbb}
\nc{\Ctt}{\C((t))}
\nc{\Ct}{\C[t,t\inv]}

\nc{\ghat}{\wh{\g}}

\nc{\un}{\underline}
\nc{\mc}{\mathcal}
\nc{\BB}{{\mc B}}
\nc{\bb}{{\mf b}}
\nc{\kk}{{\mf k}}
\nc{\frob}{\times}
\nc{\sm}{\setminus}
\nc{\Pp}{{\mathbb P}^1}
\nc{\Aa}{{\mc A}}

\nc{\AutO}{\on{Aut}\Oo}
\nc{\AUTO}{\un{\on{Aut}}\Oo}
\nc{\AUTK}{\un{\on{Aut}}\K}
\nc{\Heout}{\He_{\out}}
\nc{\Hetil}{{\widetilde\He}}
\nc{\wb}{\overline}

\nc{\Res}{\on{Res}}
\nc{\pitil}{\Pi}
\nc{\Ctil}{\wt{C}}
\nc{\auto}{\on{Aut} \Oo}
\nc{\phitil}{\wt{\phi}}
\nc{\gz}{\g_{\vec z}}
\nc{\tensorM}{\bigotimes_{i=1}^N{\mathbb M}_i}
\nc{\tensorW}{\bigotimes_{i=1}^N W_{\nu_i,k}}
\nc{\out}{\on{out}}

\nc{\m}{{\mathfrak m}}

\nc{\gx}{\g^0_{\vec x}}

\nc{\hx}{\He^0_{\vec x}}
\nc{\tensorpi}{\pi_{\nu_1,\ldots,\nu_N}^\kappa}
\nc{\Phizw}{\Phi_{\vec w}({\vec z})}
\nc{\Pro}{{\mathbb P}}

\nc{\De}{D}

\nc{\us}{\underset}

\nc{\Ll}{\mc L}
\nc{\dR}{\on{dR}}

\nc{\T}{{\mc T}}

\nc{\Xn}{\overset{\circ}X{}^n} \nc{\Dn}{\overset{\circ}D{}^n}
\nc{\Dxn}{\overset{\circ}D{}^n_x} \nc{\varphitil}{\wt{\varphi}}

\nc{\lf}{{\mf l}}
\nc{\GL}{{}^L G}
\nc{\Vir}{\on{Vir}}

\nc{\bfgn}{{\bf g}_n}
\nc{\bfzn}{{\bf z}_n}

\begin{document}

\title[Cosimplicial meromorphic functions cohomology on complex manifolds]  
{Cosimplicial meromorphic functions cohomology on complex manifolds}  
                                
\author{A. Zuevsky} 
\address{Institute of Mathematics \\ Czech Academy of Sciences\\ Zitna 25, Prague\\ Czech Republic}

\email{zuevsky@yahoo.com}

\begin{abstract}
Developing ideas of \cite{Fei}, 
 we introduce canonical cosimplicial cohomology of meromorphic functions
for infinite-dimensional Lie algebra formal series 
 with prescribed analytic behavior  
on domains of a complex manifold $M$.  
Graded differential cohomology of a sheaf of Lie algebras $\mathcal G$  
via the cosimplicial cohomology of $\mathcal G$-formal series for any covering by Stein spaces on $M$ is computed.  
A relation between cosimplicial cohomology (on a special set of open domains of $M$)  
of formal series of an infinite-dimensional Lie algebra $\mathcal G$ 
and singular cohomology of auxiliary manifold associated to a $\mathcal G$-module is found.   
Finally, multiple applications in conformal field theory,  
 deformation theory, and in the theory 
of foliations are proposed. 

\bigskip 
AMS Classification: 53C12, 57R20, 17B69 
\end{abstract}

\keywords{Meromorphic functions, cosimplicial cohomology, complex manifolds}
\vskip12pt  

\maketitle

\section{Introduction}
The continuous cohomology of Lie algebras of ${C}^{\infty}$-vector fields \cite{BS, Fuk, FF1988} has proven to be
a subject of great geometrical interest.  
There exists the natural problem of calculating the continuous 
cohomology of 
meromorphic structures on 
complex manifolds \cite{Fei, wag, Khor, Kaw, Hae, BS}.  
In \cite{Fei} Feigin obtained various results concerning (co)homologies of certain Lie algebras 
associated to a complex curve $M$.   
For the 
 Hodge decomposition  of the tangent bundle complexification of $M$  
corresponding 
Lie bracket in the space of holomorphic 
vector fields extends to a differential Lie superalgebra structure on the Dolbeault complex. 
This 
is called the differential Lie superalgebra $\Gamma(Lie(M))$ of holomorphic vector fields on $M$. 
The 
Lie algebra of holomorphic vector fields $Lie_{D}(M)$  
is defined as the cosimplicial object in the 
category of Lie algebras obtained from a covering   
of $M$ by associating to any $i_1<i_2< \ldots <i_l$,  
the Lie algebra of holomorphic vector fields $Lie( U_{i_1} \cap  \ldots  \cap U_{i_l})$.   
%
In \cite{Fei} the author calculates the continuous (co)homologies 
with coefficients in certain one-dimensional representations $\tau_{c_{p,q}}$ of these Lie (super)algebras 
where $c_{p,q}$ denotes the value of the central charge for corresponding Virasoro algebra. 
 The main 
 result 
states that 
$H_0(Lie_D(M), \tau_{c_{p,q}} )$ is isomorphic to  $H(M,p,q)$, where 
 the representation $\tau_{c_{p,q}}$ 
is derived from a vacuum representation of the Virasoro algebra, 
 and $H(M,p,q)$ is the modular functor for the minimal conformal field theory \cite{DiMaSe}. 
The algebra 
$H^*_c(\Gamma(Lie(M)))$ of continuous cohomologies acts naturally on $H_*(Lie_D(M),\tau_{c_{p,q}})$,   
and the dual space is a free $H^*_c(\Gamma(Lie(M)))$-module 
with generators in degree zero.

The paper \cite{wag} 
continues the work of Feigin 
\cite{Fei} and Kawazumi 
\cite{Kaw}
 on the Gelfand-Fuks cohomology of the Lie algebra of
holomorphic vector fields on complex manifolds.  
%
To enrich the cohomological structure, one has to involve  
cosimplicial and graded differential Lie algebras  
 well known in Kodaira-Spencer deformation theory.
The idea to use cosimplicial spaces to study the cohomology of mapping spaces 
goes back at least to 
 Anderson \cite{A},   
and it was further developed 
in \cite{BS}.
In \cite{wag}
 they compute the 
corresponding cohomologies
for arbitrary complex
manifolds up to calculation of 
cohomology of 
sections spaces of complex bundles on extra manifolds.  
 The results obtained are 
 very similar to the results of Haefliger \cite{Hae} and \cite{BS} 
in the case of ${C}^{\infty}$ vector fields. 
Following constructions of \cite{Fei}
applications in conformal
field theory (for Riemann surfaces), deformation theory, and foliation theory were proposed. 
In addition to that, in \cite{wag} 
the Quillen functor scheme was used for  
 the sheaf of holomorphic 
vector fields 
on a complex manifold, and its fine
resolution was given by the sheaf 
of $d\bar{z}$-forms with
values in holomorphic vector fields, the sheaf of Kodaira-Spencer 
algebras.

  Let $M$ be a smooth compact manifold and  $Vect(M)$ be the Lie algebra of vector fields on $M$.   
Bott and Segal \cite{BS} 
proved that the Gelfand-Fuks cohomology $H^*(Vect(M))$ 
 is isomorphic to the singular cohomology $H^*(E)$ of the space $E$ of continuous cross sections of a certain 
fibre bundle $\mathcal E$ over $M$. 
Authors of \cite{PT, Sm} continued to use advanced topological methods for more general cosimplicial spaces of maps.   

The main purposes of this paper are: 
 to compute the cosimplicial version of cohomology of  
meromorphic functions with prescribed analytic behavior on domains of arbitrary complex manifolds, 
and to find relations with 
other types of cohomologies.   
We also propose applications in conformal field theory, deformation theory, cohomology and characteristic 
classes of foliations on smooth manifolds. 
 
As it was demonstrated in \cite{wag}, the ordinary 
 cohomology of vector fields on complex manifolds turns to be not the most effective and general one.   
In order to avoid trivialization 
and reveal a richer cohomological structure  
of complex manifolds cohomology, one has to 
treat \cite{Fei}  holomorphic vector fields as a sheaf rather than taking 
global sections. 
%
Inspite results in previous approaches, it is desirable to     
  find a way to enrich cohomological structure which motivates 
 construction of more refined cohomology description for non-commutative 
algebraic structures.  
The idea of meromorphic function cosimplician cohomology for complex manifolds was outlined in \cite{Fei} in 
conformal field theory form (for Riemann surfaces)  
and is developing in this paper. 
In particular, 
 we study relations of the sheaf 
of meromorphic functions associated to certain Lie algebras   
to the sheaf $\mathfrak{g}$  
of vector valued
differential forms.  
\section{Meromorphic functions with prescribed analytic behavior} 
In this section the space underlying cohomological complexes is defined in terms of meromorphic 
functions with certain properties \cite{H2, Huang}, in particular, prescribed analytic behavior. 
In order to shorthand expressions and notations, we call such functions mero functions. 
Mero functions depend implicitly on an infinite number of non-commutative parameters via 
Lie-algebra valued formal series in several complex variables. 

\subsection{Algebraic completion of the space of formal $\mathcal G({\bf z})$-module valued series}
In this subsection we describe the algebraic completion of the space of an infinite-dimensional 
Lie algebra $\mathcal G$ module-valued formal series.  
In the whole body of the paper we use the notation 
 ${\bf y}_{(n, 1), \ldots,  (n, k)}$, ${\bf y}_n=(y_1, \ldots, y_n)$,  for $n \ge 0$ 
sets of $k_i$, $1 \le i \le n$ of $k_i$ variables.  
If $n$  and $k_i$ are not explicitly specified we write ${\bf y}$. 
Let $\mathcal G$ be an infinite-dimensional Lie algebra  
 generated by $g_i$, $i \in \Z$.  
%
Let $\mathcal G({\bf z})$ be the space of power series in several complex formal variables.  
In order to define a specific space of meromorphic functions associated to $\mathcal G$ 
and satisfying certain properties mentioned later on, 
we have to (as in \cite{Huang}) work with the algebraic completion of a $\mathcal G$-module $W$ \cite{Huang}. 
 In particular, for that purpose, we have to consider 
elements of a $\mathcal G$-module $W$ with inserted 
exponentials of the grading operator $L(0)$, i.e., of the form 
$\sum\limits_{n \in \Z} z^{-n-1}  \; a^{L(0)} v_n \; b^{L(0)} v$. 
For general $a$, $b$, $z \in \C^\times$, 
 such elements  
do not satisfy the properties (as elements of $W$ do) needed  to construct a closed theory of meromophic functions with 
prescribed properties. 
Thus we have to extend $W$ algebraically (and analytically), i.e., include extra elements  
to make the structure of $W$ compatible with the descending filtration with respect to the grading subspaces, 
and analytic properties with respect to formal parameter $z$. 
Thus we end up with are elements of the algebraic completion $\overline{W} = \prod_{n\in \Z} W(n)$, of $W$, 
that has the structure that is complete in the
 topology determined by the filtration and the residue pairing above. 
Recall that 
$W'=\coprod_{n\in \mathbb{Z}}W_{(n)}^{*}$.  
For a infinite-dimensional Lie algebra $\mathcal G$-module $W$, 
let $G$ be the algebraic completion of the space of $W$-valued formal series 
 $G=\prod_{n\in \C} W_{(n)}=(W')^{*}$,  
endowed with a complex grading (with respect to a grading operator $K_G$).  
We assume that on the space of formal series associated to $\mathcal G$, 
 there exists a non-degenerate bilinear pairing $( ., .)$,  
and matrix elements of $G$ elements are given by this pairing. 
In addition to that, we consider an analytical extension with respect to the pairing 
${\rm Res}_z ( ., .  )$ for $G$. 
\subsection{Meromorphic functions with non-commutative parameters}
\label{functional}
%
In this subsection we give axiomatic definition of meromorphic functions with specific properties. 
Let $M$ be an $n$-dimensional smooth complex manifold. 
Let ${\bf p}_l$ be a set of $l$ points on $M$.
For each $p_i$, $1 \le i \le l$, let $U_i$
 be an open domain surrounding $p_i$. 
Let us identify formal variables ${\bf z}_{n,l}$ with $l$ sets of $n$ local coordinates on $\U_l$. 
In this paper we consider 
 meromorphic functions 
 of several complex variables  
 defined on sets of open  
 domains of $M$ with local coordinates ${\bf z}_{l}$    
 which are extandable to meromorphic functions  
 on larger domains on $M$. We denote such extensions by $R(f({\bf z}_{l}))$.   
Denote by $F_{n}\mathbb C$ the  
configuration space of $l \ge 1$ ordered coordinates in $\mathbb C^{ln}$, 
\[
F_{ln}\mathbb C=\{{\bf z}_l \in \mathbb C^{ln}\;|\; z_{i, l} \ne z_{j, l'}, i\ne j\}.
\] 
We assume that there exists a non-degenerate bilinear pairing $( .,. )$ on $G$, and 
 denote by $\widetilde G$ the space dual with respect to this pairing.   
In order to work with objects having coordinate invariant formulation \cite{BZF},  
for a set of $G$-elements ${\bf g}_{l}$  
we consider converging 
meromorphic  
functions $f({\bf x}_{l})$ of ${\bf z}_{l} \in F_{ln}\mathbb C$, 
with ${\bf x}_{l}= ({\bf g}_{l}, {\bf z}_{l} {\bf dz}_{l})$,  
where ${\bf z}_{l}$ are multiplied by corresponding differentials ${\bf dz}_{l}$. 
\begin{definition}
For arbitrary $\vartheta \in \widetilde G$,  
we call 
 a map linear in ${\bf g}_l$ and ${\bf z}_l$, 
\begin{equation}
F: {\bf x}_l
\mapsto   
\label{deff}
    R(\vartheta, f({\bf x}_l 
)), 
\end{equation}
  a meromorphic function in ${\bf z}_l$   
with the only possible poles at 
$z_{i, l}=z_{j, l'}$, $i\ne j$. 
Abusing notations, we denote $F({\bf x}_l)= R(\vartheta, f({\bf x}_l))$.  
\end{definition}
\begin{definition}
We define  left action of the permutation group $S_{ln}$ on $F({\bf z}_l)$ 
by
\[
\nc{\bfzq}{{\bf z}_l}
\sigma(F)({\bf x}_l)=F({\bf g}_l, {\bf z}_{\sigma(i)}).  
\]
\end{definition}
In particular, meromorphic functions described above can be realized as $R(( ., .))$
 of the bilinear pairing. 
\subsection{Conditions on meromorphic functions} 
Let ${\bf z}_l \in F_{ln}\C$.  
Denote by $T_G$ the translation operator \cite{K}. 
We define now extra conditions leading to the definition of restricted  
meromorphic functions. 
\begin{definition}
\label{tupo1}
Denote by $(T_G)_i$ the operator acting on the $i$-th entry.
 We then define the action of partial derivatives on an element $F({\bf x}_l)$  
\begin{eqnarray}
\label{cond1}
\partial_{z_i} F({\bf x}_l)  &=& F((T_G)_i \; {\bf g}_l, {\bf z}_l),   
\nn
\sum\limits_{i \ge 1} \partial_{z_i}  F({\bf x}_l)  
&=&  T_{G} F({\bf x}_l),   
\end{eqnarray}
\end{definition}
and call it $T_{G}$-derivative property. 
\begin{definition} 
\label{tupo2}
For   $z \in \C$,  let 
\begin{eqnarray}
\label{ldir1}
 e^{zT_G} F ({\bf x}_l)   
 = F({\bf g}_l, {\bf z}_l +z). 
\end{eqnarray}
 Let 
 ${\rm Ins}_i(A)$ denotes the operator of multiplication by $A \in \C$ at the $i$-th position. Then we define   
\begin{equation}
\label{expansion-fn}
F({\bf g}_l, {\rm Ins}_i(z) \; {\bf z}_l)=  
F( {\rm Ins}_i (e^{zT_G}) \; {\bf g}_l, {\bf z}_l), 
\end{equation}
are equal as power series expansions in $z$, in particular, 
 absolutely convergent
on the open disk $|z|<\min_{i\ne j}\{|z_{i, l}-z_{j, l'}|\}$. 
\end{definition}
\begin{definition}
A meromorphic function has $K_G$-property   
if for $z\in \C^{\times}$ satisfies 
$z{\bf  z}_l \in F_{ln}\C$,  
\begin{eqnarray}
\label{loconj}
z^{K_G 
} F ({\bf x}_l) = 
 F (z^{K_G} {\bf g}_l, 
 z\; {\bf z}_l). 
\end{eqnarray}
\end{definition}

\subsection{Meromorphic functions with prescribed analytical behavior}
In this subsection we give the definition of meromorphic functions with prescribed analytical behavior
on a domain of complex manifold $M$ of dimension $n$.
To shorthand, we call such functions mero functions.     
We denote by $P_{k}: G \to G_{(k)}$, $k \in \C$,      
the projection of $G$ on $G_{(k)}$.
For each element $g_i \in G$, and $x_i=(g_i, z)$, $z\in \C$ let us associate a formal series  
$W_{g_i}(z)= W(x_i)=  
\sum\limits_{k \in \C }  g_{i} \; z^{k} \; dz $, $i \in \Z$. 
  Following \cite{Huang}, we formulate 
\begin{definition}
\label{defcomp}
We assume that there exist positive integers $\beta(g_{l', i}, g_{l", j})$ 
depending only on $g_{l', i}$, $g_{l'', j} \in G$ for 
$i$, $j=1, \dots, (l+k)n $, $k \ge 0$, $i\ne j$, $ 1 \le l', l'' \le n$.  
 Let   
${\bf l}_n$ be a partition of $(l+ k)n     
=\sum\limits_{i \ge 1} l_i$, and $k_i=l_{1}+\cdots +l_{i-1}$. 
For $\zeta_i \in \C$,  
define 
$h_i  
=F 
({\bf W}_{ { \bf g}_{ k_i+{\bf l}_i  }} ( 
 {\bf z}_{k_i + {\bf l}_i  }- \zeta_i ))$, 
for $i=1, \dots, ln$.
We then call a meromorphic function $F$ satisfying properties \eqref{cond1}--
\eqref{loconj},
a meromorphic function with prescribed analytical behavior, of a mero function, if 
under the following conditions on domains, 
\[
|z_{k_i+p} -\zeta_{i}| 
+ |z_{k_j+q}-\zeta_{j}|< |\zeta_{i} -\zeta_{j}|,  
\]
for $i$, $j=1, \dots, k$, $i\ne j$, and for $p=1, 
\dots$,  $l_i$, $q=1$, $\dots$, $l_j$, 
the function   
%
 $\sum\limits_{ {\bf r}_n \in \Z^n}  
F( {\bf P_{r_{i}}  h_i}; (\zeta)_{l})$,     
is absolutely convergent to an analytically extension 
in ${\bf z}_{l+k}$, independently of complex parameters $(\zeta)_{l}$,
with the only possible poles on the diagonal of ${\bf z}_{l+k}$   
of order less than or equal to $\beta(g_{l',i}, g_{l'', j})$.   
 In addition to that, for ${\bf g}_{l+k}\in G$,  the series 
$\sum_{q\in \C}$  
$F( {\bf W(g_k}$, ${\bf P}_q ( {\bf W(g}_{l+k}, {\bf z}_k), {\bf z}_{ k + {\bf l} }))$,    
is absolutely convergent when $z_{i}\ne z_{j}$, $i\ne j$
$|z_{i}|>|z_{s}|>0$, for $i=1, \dots, k$ and 
$s=k+1, \dots, l+k$ and the sum can be analytically extended to a
meromorphic function 
in ${\bf z}_{l+k}$ with the only possible poles at 
$z_{i}=z_{j}$ of orders less than or equal to 
$\beta(g_{l', i}, g_{l'', j})$. 
\end{definition}
For $m \in \N$ and $1\le p \le m-1$, 
 let $J_{m; p}$ be the set of elements of 
$S_{m}$ which preserve the order of the first $p$ numbers and the order of the last 
$m-p$ numbers, that is,
$$J_{m, p}=\{\sigma\in S_{m}\;|\;\sigma(1)<\cdots <\sigma(p),\;
\sigma(p+1)<\cdots <\sigma(m)\}.$$
Let $J_{m; p}^{-1}=\{\sigma\;|\; \sigma\in J_{m; p}\}$.  
In addition to that, for some meromorphic functions require the property: 
\begin{equation}
\label{shushu}
\sum_{\sigma\in J_{ln; p}^{-1}}(-1)^{|\sigma|} 
\sigma( 
F ({\bf g}_{\sigma(i)}, {\bf z}_l) 
)=0. 
\end{equation}
Finally, we formulate 
\begin{definition}
\label{poyma}
We define the space $\Theta(ln, k, U)$ of all $l$ complex $n$-variable restricted    
 meromorphic functions 
 with prescribed analytical behavior 
on a $F_{ln}\C$-domain $U \subset M$  and   
satisfying 
 $T_G$- and $K_G$-properties \eqref{cond1}--
\eqref{loconj}, 
 (definition \eqref{defcomp},  
  and \eqref{shushu}).   
\end{definition}
\section{Properties of cosimplicial double complex spaces} 
\label{cohass}
In this section we define the double complexes of mero 
function cohomology 
on a complex manifold $M$ of complex dimension $n$. 
\subsection{Spaces of cosimplicial double complexes} 
In \cite{GF} the original approach to cohomology of vector fields of manifolds 
was initiated.   
 Another approach  to cohomology   
 of the Lie algebra of vector fields on a manifold in the cosimplicial setup we find  
in \cite{Fei, wag}.   
Let $\U$ be a covering $\left\{ U_j \right\}$ on $M$, 
and ${\bf z}_{l, j}$ be 
$l$ sets of local complex coordinates on each domain $U_j$ 
around $l$ points ${\bf p}_{l, j}$. 
For a set of $G$-elements  
${\bf g_l}$,   
and differentials ${\bf dz}_{l, j}$,  
we consider 
 ${\bf x}_{l, j}= ({\bf g}_l, {\bf z}_{l, j} \;{\bf dz}_{l, j})$.  
\begin{definition}
For a domain $U \subset M$, and 
 $l$, $k \ge 0$,    
we denote by ${C}^{l}( \Theta(ln, 1), U)$       
the space of all mero 
functions $\Theta(ln, 1)$ with prescribed analytic behavior with respect to 
  $l$ sets of $n$ complex coordinates introduced on $U$.
\end{definition}
\begin{remark}
Note that according to our construction, $M$ can be infinite-dimensional. Thus, in that case, 
 we consider $l$ infinite sets of complex coordinates. 
 The set of $ln$ $G$-elements 
${\bf g_{l} }$ 
plays the role of non-commutative parameters in our cohomological construction. 
\end{remark}
Using the standard method of defining canonical (i.e., independent of the choice  
of covering $\U$) cosimplicial object \cite{Fei, wag}, 
 we consider mero 
functions 
$F ( {\bf x}_{l, j} )$ on $\U$,   
 and give the following definition of a general cosimplicial double complex for $\mathcal G$.   
\begin{definition}
Choose a covering $\U=\left\{U_i, i \in I\right\}$ of $n$-dimensional complex manifold $M$. 
Let us associate to any subset
$\left\{i_1< \cdots < i_k\right\}$ of $I$, 
the space of restricted meromorphic functions  
converging on the intersection 
$\left\{U_{i_1}\cap \ldots \cap U_{i_k}\right\}$. Let us introduce the space 
\begin{equation}
\label{ourbi-complex} 
 {C}^{l}_{k  }(G, \U) =  
 C^{l}
\left(  \Theta(ln, kn), \; \bigcap_  
 {
i_1 \le \ldots \le i_k, k \ge 0 } U_{ i_k} 
\right). 
\end{equation}
 We call this space a cosimplicial cohomology object in the category of algebras of 
mero functions on $M$.  
\end{definition}
\subsection{Co-boundary operators}
\label{coboundaryoperator}
Let us take 
${C}_{k}^{0}
= G$. 
 Then we have 
\[
 {C}_{k}^{l}
\subset  {C}_{k-1}^{l}, 
\]
when lower index is zero the sequence terminates.   
We also define 
\[
C_{\infty}^{n}= \bigcap_{m\in \N}C_{m}^{n}. 
\]
We organize elements of ${\bf x}_{(n,l)} =\left({\bf x}_{n, 1} \ldots, {\bf x}_{n,l} \right)$ 
 as $l$ groups by $n$ elements. 
Denote by ${\bf x}_{(n,\widehat {i}, \ldots, \widehat{j}) }$, $1 \le i, j \le l$, 
the set ${\bf x}_{(n, l)}$ with ${\bf x}_{n, i}, \ldots, {\bf x}_{n,j}$ omitted. 
For 
$F \in {C}_{k }^{l}$, 
we define  
the operator $D^{l }_{k }$ by   
\begin{eqnarray}
\label{hatdelta}
{D}^{l }_{k } F({\bf x}_{(n,l)}) &=& T_1(W_W({\bf x}_{n, 1} )).F({\bf x}_{(n, \widehat {1}) })     
\nn
&+&\sum_{i=1}^{l }(-1)^{i} 
T_i( W_V({\bf x}_{n, i})  )T_{i+1}(W_V({\bf x}_{n, i+1})).
F\left( {\bf x}_{(n, \widehat{i}, \widehat{i+1} )  }  
  \right)       
\nn
 &+&(-1)^{l+1}  
 T_1(W_W({\bf x}_{n, l+1})).F\left( {\bf x}_{n, \widehat {l+1}}  \right), 
\end{eqnarray}
where $T_i(\gamma).$ denotes insertion of $\gamma$ at $i$-th position of $F$. 

In \cite{Huang} we find the construction of double chain-cochain complex for a space of 
meromorphic functions compatible with a few formal series as in definition \ref{defcomp}. 
In particular, (c.f. Proposition 4.1), the chain condition for such double complex is proven.  
Here we use that construction to prove the chain property for the operator \eqref{hatdelta} 

\begin{proposition}
\label{cochainprop}
The operator \eqref{hatdelta} 
 forms the double chain complex 
\begin{equation}
\label{conde}
  D^{l }_{k  }:  {C}_{k  }^{l }  
\to  {C}_{k-1 }^{l+1},  
\end{equation}  
\begin{equation*}
   D^{l+1 }_{k-1  } \circ {D}^{l }_{k  }=0, 
\end{equation*} 
on the spaces \eqref{ourbi-complex}.
\end{proposition}
\begin{definition}
According to this proposition,  
 one defines the 
 $(l, k)$-th mero 
function cosimplicial cohomology $H^{l}_{k, \; cos}(\mathcal G, \U)$ of $M$  
to be 
 $H_{k, \; cos}^l(\mathcal G, \U) ={\rm Ker} \; 
D^l_k/\mbox{\rm Im}\; D^{l-1}_{k+1  }$. 
\end{definition}
\begin{proof}
 We use the proof of Proposition 4.1 of \cite{Huang}. 
In that proof, for the case $n=1$, it was shown that the coboundary operator
 $D^l_k$ acting on elements of $C^l_k$
brings about elements of $C^{l+1}_{k-1}$.  
Namely, for $n=1$ and $F \in \Theta(l+1, m)$ one has  
\begin{eqnarray*}
 D^{l}_{m} F\left(x_1, \ldots, x_{l+1} \right)    
&=&   W_W(x_1) \;  F(x_2, \ldots,  x_{l+1})   
 \nn 
&+&  \sum_{i=1}^{l}(-1)^{i}    
 F(x_{1}, \ldots, x_{i-1}, 
 \; W_{V}(v_{i}, z_{i}-\zeta_{i}) W_{V}(v_{i+1}, z_{i+1}-\zeta_{i}), \;    
\nn
&& \quad\quad\quad\quad
x_{i+2}, \ldots, x_{l+1})  
\ + (-1)^{l+1} W_W(x_{l+1}) \; 
 F(x_{1},  \ldots,  x_{l}),  
\end{eqnarray*}
 By Proposition 2.8 
of \cite{Huang},   
$D^{l}_{k}F$ is  
compatible with $(k-1)$ formal series as in definition \ref{defcomp} ,and has the $T_G$-derivative  \eqref{cond1}  
property and the $K_G$-conjugation  \eqref{ldir1} properties according to definitions \ref{tupo1} and \ref{tupo2}.  
So $D^{l}_{k}F \in 
C_{k-1}^{l+1}$ and $D^l_{k}$ is indeed a map whose image is in 
${C}_{k-1}^{l+1}$. 
The chain property for $D^l_k$ was proven also. 
Applying inductively Proposition 4.1 of \cite{Huang} to our setup and increasing recursively for $n\ge 0$, 
we obtain the result of our proposition. 
\end{proof}
\section{Sheaf formulation of cosimplicial cohomology} 
In this section we generalize the construction of \cite{GF, Fei, wag} and  
replace the algebra of holomorphic vector field used in \cite{wag} 
with $G$-valued series for an infinite-dimensional  
Lie algebra $\mathfrak g$ in the sheaf-theoretical formulation. 
In particular we prove a generalization of Theorem 4 of \cite{wag} and  
relate the graded differential cohomology of the sheaf of $G$-valued formal series 
and the cosimplicial cohomology of the sheaf of mero functions over Stein sets on $M$. 
Let us start with some general definitions needed for further explanations. 
We understand from sections that $G$-valued series for $\mathfrak g$ are formalized via  
the space of mero functions. 
Thus it makes sense to formulate the definition of the sheaf of graded differential 
Lie algebras of $G$-valued formal series for $\mathfrak g$ simultaneously with 
the definition of the sheaf of mero functions.  
Moreover, as we will see in one of next subsections, there exists an equivalence of the cohomology of 
the sheaf of graded differential 
Lie algebras of $G$-valued formal series for $\mathfrak g$  
and cosimiplicial cohomology of a complex for a sheaf of $G$-valued series for $\mathfrak g$.  

%
To any $U$ of $\U$ of $M$ we can assign the set $\Theta(ln, kn, U)$ of mero functions on $U$.  
 The restriction maps are then just given by 
restricting a mero function on $U$ to a smaller open subset $\widetilde{U} \subset U$,  
which, according to definition \ref{defcomp} is again a mero function.
We then immediately check the presheaf axioms. 

\subsection{Coherent sheaf of mero functions} 
$M$ is endowed  with a sheaf
 of rings $\mathcal O_X$, the sheaf of holomorphic functions or regular functions, 
and coherent sheaves are defined as a full subcategory of the category 
of 
$\mathcal O_X$-modules (that is, sheaves of 
$\mathcal O_X$-modules). 
We are able
localize mero functions 
to open subsets
 $U \subset M$.  
The presheaf of mero functions 
 can be glued to global data. 
We formulate the following 
Denote by ${\mathcal F}_M$ the coherent sheaf of mero  
functions on $M$, and consider ${\mathcal F}_M$-modules.  
The sheaf of mero functions on $M$ is defined by 
and ${\mathcal F}_M$-modules. 
As in \cite{wag} we transfer to the sheaf setup of $G$-valued series and mero functions. 
Let us denote by ${\mathcal O}_M$ the coherent sheaf of 
holomorphic functions on $M$ and by ${\mathcal E}_M$ the sheaf of
$C^{\infty}$ functions on $M$.
We denote by $\F_M$ the sheaf of mero functions
on $M$, and 
 by ${\mathfrak f}_M$ the sheaf of $\F_M$-modules. 
Denote by ${\mathcal G}({\bf z})$ the sheaf of Lie-algebra $\mathfrak g$ $n$-parameter ${\bf z}$ formal series. 
 It can be represented via by ${\mathcal F}_M$-modules for the sheaf of mero 
functions. 
Note that the sheaf $\mathcal G({\bf z})$ naturally induces the sheaf ${\mathcal E}_M$. 
 Let ${\mathfrak F}$ a sheaf of ${\mathcal O}_X$-modules which are Lie algebras.
Let
$(\mathfrak{F},\bar{\partial})$ a sheaf of differential graded Lie algebras which are ${\mathcal E}_M$-modules.  
We denote by $\Ga(\mathfrak{F})$, $\Ga(X,\mathfrak{F})$ or 
$\mathfrak{F}(X)$ the differential graded Lie algebras of global sections of the sheaf 
$\mathfrak{F}$.
As in \cite{wag}, 
 we can associate to $\mathfrak{F}$ resp. to
$(\mathfrak{F},\bar{\partial})$ sheaves of differential graded coalgebras
$C_*(\mathfrak{F})$, $C_{*,dg}(\mathfrak{F})$, $H_*(\mathfrak{F})$ and
$H_{*,dg}(\mathfrak{F})$ where the last two carry the trivial
differential.
 In the same way, we have sheaves of differential graded
algebras $C^*_{cont}(\mathfrak{F})$, $C_{dg}^*(\mathfrak{F})$,
$H^*_{cont}(\mathfrak{F})$ and $H_{dg}^*(\mathfrak{F})$. 
Furthermore, we have differential graded coalgebras $C_*(\Ga(\mathfrak{g}))$, 
$C_{*,dg}(\Ga(\mathfrak{F}))$, $H_*(\Ga(\mathfrak{F}))$ and
$H_{*,dg}(\Ga(\mathfrak{F}))$, and the corresponding algebras. 
For a sheaf $\A$ 
on   $M$, the sheaf cohomology groups
$H^{i}(\A, M)$  
for integers $i$
are defined as the right derived functors of the functor of global sections, 
$\A \mapsto \A(X)$.  
As a result, 
${\displaystyle H^{i}(\A, M)}$ is zero for 
${\displaystyle i<0}$,  
and 
${\displaystyle H^{0}(X, \A)}$ can be identified with 
${\displaystyle \A}(X)$. 
%
For any short exact sequence of sheaves 
${\displaystyle 0\to {\mathcal {A}}\to {\mathcal {B}}\to {\mathcal {C}}\to 0}$,
 there is a long exact sequence of
 cohomology groups: 
\[
{\displaystyle 0\to H^{0}(X,{\mathcal {A}})\to H^{0}(X,{\mathcal {B}})\to H^{0}
(X,{\mathcal {C}})\to H^{1}(X,{\mathcal {A}})\to \cdots .}
\]
\subsection{Thickened nerve of the covering}  
Here we have to provide some further topological definitions.  
%
The nerve of an open covering is a construction of an abstract simplicial complex
 from an open covering of $M$. 
\begin{definition}
Let 
$I$ be an index set and $\U$ be a family of open subsets 
 $U_{i}$ of $M$ indexed by $i\in I$.  
The nerve of $\U$ is a set of finite subsets of the index-set $I$.  
It contains all finite subsets $J\subseteq I$ such that the intersection of the 
 $U_{i}$ whose subindices are in $J$ is non-empty, i.e., 
$N(\U)= \left\{ J\subseteq I:\bigcap _{j\in J} U_{j} \neq \varnothing \right\}$, where $J$ 
is a finite set.    
\end{definition}
\begin{definition}
If $J\in N(\U)$, then any subset of $J$   
is also in $N(\U)$, 
 making $N(\U)$ an abstract simplicial complex,   
often called the nerve complex of $C$.
\end{definition}
Next we recall the notation of a manifold made out of simplices: 
A simplicial manifold is a simplicial complex for which the geometric realization 
is homeomorphic to a topological manifold. 
This is essentially the concept of a triangulation in topology.  
This can mean simply that a neighborhood of each vertex 
(i.e., the set of simplices that contain that point as a vertex) is homeomorphic to a $n$-dimensional ball.
%
For our further purposes we will need the following proposition from \cite{BS}: 
\begin{proposition}
(Proposition (5.9)  of \cite{BS}).
 If $C \to C'$ is a morphism of simplicial cochain complexes such that
$C_p \to C_p'$ is a cohomology equivalence for each $p$, then $|C| \to  |C'|$ is a cohomology equivalence. 
\end{proposition}
\subsection{Definition of the cosimplicial sheaf complex} 
In this subsection, following the general ideas of
\cite{CM, wag} we formulate definition of a cohomology $H^*(\U, \A)$ of  
a cosimplicial sheaf $\A$ complex $C_*(\U, \A)$ defined  
 on a system $\U$ of domains on smooth complex manifold $M$.   
\begin{definition}
Let $\Gamma (\A, U)$ be a local section defined on $U$ for a coherent sheaf $\A$. 
Then one defines a chain-cochain complex 
\begin{equation}
\label{sheaf-complex} 
 {C}_{k}(\U, \A) =  
 \Gamma \left(\A, \; \bigcap_    
 {
i_1 \le \ldots \le i_k, k \ge 0,  \atop  U_{0}\stackrel{h_1}{\rmap} \ldots \stackrel{h_k}{\rmap} U_k
} U_{ i_k}\right). 
\end{equation}
where the intersection is over all $k$-strings of $h$-
embeddings between opens $U_i\in \U_{i+1}$,  
and the boundary operator $\delta = \sum(-1)^i\delta_i$ is given by the standard formula
\begin{equation}\label{deltas} \delta_{i}\Gamma(h_1, \ldots , h_{k+1})= \left\{ \begin{array}{lll}
                                      h_{1}^{*}\Gamma(h_2, \ldots , h_{k+1}), \ \ \mbox{if $i=0$}\\ 
                                      \Gamma(h_1, \ldots, h_{i+1}h_{i}, \ldots, h_{k+1}), \ \ \mbox{if $0<i< k+1$}\\
                                      \Gamma(h_1, \ldots, h_k), \ \ \mbox{if $i= k+1$}
                        \end{array}
                \right.
\end{equation}
originating from application of the Chevalley-Eilenberg \cite{CE} functor. 
\end{definition}
Note that according to \cite{CM}, one has 
\begin{lemma}
\label{lemmaCM}
The complex \eqref{sheaf-complex} is a graded differential algebra with the usual product 
\begin{equation}
\label{multiplication}
(\Gamma_1\cdot\Gamma_2)(h_1, \ldots , h_{k+k'})= (-1)^{kk'}\Gamma_1(h_1, \ldots , h_{k})\; (h_{1}^{*}
 \ldots h_{k}^{*}) \;\Gamma_2(h_{k+1}, \ldots h_{k+k'}), 
 \end{equation}
for $\Gamma_1\in C_{k}(G, \U)$ and $\Gamma_2\in C_{k'}(G, \U)$.
\end{lemma}

\begin{definition}
For $\A=\F_M(\Theta(ln, kn), \U)$ beeing local $\F_M$-section defined on $U$, 
we introduce the cosimplicial sheaf complex of mero functions by \eqref{sheaf-complex}. 
\end{definition}
\subsection{Definition of the cosimplicial sheaf cohomology $H^l_{k, \; cos}(\mathcal G(\bf z), \U)$  
for the sheaf of mero functions on $M$} 
Next we define
\begin{definition}
$H_*^{*}( \G({\bf z)}, \U)$ as the cohomology of the double complex ${C}_k( \es^{l}, \U)$ given by 
\eqref{sheaf-complex} 
, where 
\[
0\rmap \A\rmap \es^0\rmap\ldots \rmap \es^d\rmap...,  
\]
 is a bounded resolution by $\U$-acyclic sheaves. 
\end{definition}
By the usual arguments \cite{CM}, such resolutions always exist, and the definition does not depend on
 the choice of
the resolution.
Let us now recall \cite{wag} the universal notion of Stein spaces which we will use 
in further discussions. 
\begin{definition}
An open set $U\subset M$ of a complex manifold $M$ is called a Stein
open set, if 
the coherent
sheaf cohomology vanishes on $U$, i.e.,   
$H^i_k( \A, U) = 0$,  $i \ge 1$, $k \ge 0$,  
and for all coherent sheaves $\A$ on $M$. 
\end{definition}
 In the language of \cite{CM} the Stein spaces are called $\U$-acyclic. 

\subsection{Graded differential 
cohomology $H^*_{*, dg} (\Gamma(\mathcal G({\bf z})), M)$  
of global sections of the sheaf $\mathcal G({\bf z})$ } 
In this subsection we recall and make applications of certain facts \cite{wag} 
on construction of the graded differential cohomology $H^*_{*, dg} (\Gamma(\mathcal G({\bf z})), M)$ 
of global sections of the sheaf $\mathcal G({\bf z})$ of an infinite-dimensional Lie algebra $\mathfrak g$ 
formal series.  
 This 
cohomology is calculated  
by associating to $\mathfrak g$ a graded differential 
algebra $C^*(\mathfrak g) = (Hom(\Lambda^*(\mathfrak g),\C), d)$, 
the 
cohomological Chevalley-Eilenberg complex \cite{CE} described below. 
\begin{definition}
Let $\mathfrak a$ be a 
Lie algebra over $\C$. 
 The Chevalley--Eilenberg chain complex is a  
projective resolution 
$W^*(\mathfrak a) \to \C$  
of the trivial $\mathfrak a$-module $\C$ in the abelian category of $\mathfrak a$-modules   
(what is the same as $U({\mathfrak a})$-modules, where $U({\mathfrak a})$ is the universal   
enveloping algebra of $\mathfrak a$).  
Graded components of the underlying $\C$-module  
of this resolution is given by 
$W_p(\mathfrak a)= U(\mathfrak a) \otimes_k \Lambda^p \mathfrak a$,  
and it has the obvious $U({\mathfrak a})$-module structure by multiplication in the first tensor factor, 
because $\Lambda^p \mathfrak a$ is free as a $\C$-module.   
For $u \in U({\mathfrak a})$, and ${\bf v}_p \in {\mathfrak a}^{\otimes p}$,  the differential is given by
\begin{eqnarray}
d\left(u \otimes {\bf \bigwedge v}_p \right) = 
 \sum_{i = 1}^p (-1)^{i+1} u\; v_i \otimes  {\bf \bigwedge v}_{(p, \widehat {i}) }
+\sum_{i < j} (-1)^{i+j} u\otimes [v_i, v_j] \wedge {\bf \bigwedge v}_{(p, \widehat {i}, \widehat {j})},  
\end{eqnarray}
%
Then, let  
$g = (\bigoplus_{i=0}^{n}g^i,  {\partial})$,  
 be a cohomological graded differential Lie algebra (which we will denote dgla in notations). 
As it was mentioned in \cite{wag}, 
 there exist two functors,  
$C_{*,dg}$ and $C^*_{dg}$, 
associating to $(g, {\partial})$ graded differential coalgebras 
$C_{*,dg}(g)$ and $C^*_{dg}(g)$.  
$C_{*,dg}(g)$ is called the Quillen functor, \cite{Qui}.   
It was explicitly constructed in \cite{HinSch}.    
The cohomology version was 
used in 
\cite{Hae} and  
\cite{SchSta}. 
Explicitly, it is given by 
\begin{displaymath}
C_{k,dg}(g) = \bigoplus_{k=p+q}C_{dg}^p(g)^q = 
\bigoplus_{k=p+q}S^{-p}(g_{(q+1)}), 
\end{displaymath}
 as graded vector spaces. 
Here $S^{-p}(g_{(q+1)})$ 
 is the graded symmetric algebra
$S^*$ on the shifted by one graded vector space $g_{(q+1)}$.  
The differential on $C_{*, dg}(g)$ is the
direct sum of the graded homological Chevalley-Eilenberg \cite{CE} differential
in the tensor direction (with degree reversed in order to have a
cohomological differential) and the differential induced on
$S^*(g_{(q+1)})^*$ by $ 
{\partial}$. 
%
 
We denote by $\Ga(\mathfrak{F}, M)$  
 the graded differential Lie algebra of global sections of the sheaf $\mathfrak{F}$ on $M$.  
Consider the sheaf of $G$-valued series for $\mathfrak g$. It 
 constitutes a sheaf of Lie algebras. 
According to \cite{HinSch} (proposition .), for any sheaf of Lie algebras
$\mathfrak{h}$ there is another sheaf of differential
graded Lie algebras constituing a resolution of $\mathfrak{h}$.
 It is the sheaf of cosimplicial Lie algebras given by taking $\mathfrak{h}$ on
the ${\rm \check C}$ech complex \eqref{sheaf-complex} associated to a covering ${\mathcal U}$ by Stein open
sets, suitably normalized by the Thom-Sullivan functor, see \cite{HinSch, wag}.
One can see that such sheaf of differential graded Lie algebras is given by 
the sheaf $\mathcal G({\bf z})$ taken our complex \eqref{sheaf-complex} 
on $\U$ given by Stein open sets. 

\subsection{Cosimplicial cohomology 
$H_{*\; cos}^*(\check C(G, 
 \U))$} 
Developing ideas of \cite{wag, CE, S}, we 
 give the definition of the cohomology $H_{*\; cos}^*(\check C(G, \U))$
of cosimplicial Lie algebra of the complex 
 of mero functions for  
an infinite-dimensional Lie algebra $\mathfrak g$ defined on 
 a specific covering. 
\begin{definition}
 The cosimplicial cohomology $H_{*\; cos}^*(\check C(G, \U))$ of the complex    
$C^l_k(G, \U)$, $l$, $k \ge 0$, 
with $C^l_k(G, \U)$ given by the complex 
of 
mero functonsi  
for the Lie algebra $\mathfrak g$ of 
 $G$-valued formal series defined on a covering $\U$       
 %
is the cohomology of 
the realization of simplicial cochain complex 
obtained from applying the continuous (chain-cochain) 
 Chevalley--Eilenberg complex as a functor $C^*_{*, cont}$ to the 
cosimplicial Lie algebra $\check C(G, \U)$.  
\end{definition}

\subsection{Computation of graded differential sheaf cohomology
 via cohomology of cosimplicial Lie algebra of mero functions cohomology}
The main idea of this subsection is that we are able to compute 
 the graded differential algebra cohomology defined for the sheaf of $G$-valued series 
for an infinite-dimensional Lie algebra $\mathcal G$  
on an $n$-dimensional complex manifold $M$ 
 via  the mero function cosimplicial cohomology for $\mathcal G$ considered   
on special type of open domains on $M$. 

We then obtain the main result of this section 
\begin{proposition}
On a complex manifold $M$ of dimension $n$, one has 
\begin{displaymath}
 H^*_{*\; dg} \left( \Gamma(\mathcal G({\bf z}),  M)  \right)  \cong H^*_{*\; cos} \left( {\check C} (G,  
{\mathcal U})\right), 
\end{displaymath}
for any covering of $M$ by Stein open sets ${\mathcal U}$ (with respect to the cosimplicial sheaf cohomology of
coherent sheaves).  
\end{proposition}
\begin{proof}
The idea of the proof is quite close to remarks to the proof of Theorem 4 of \cite{wag}. 
Here we give an explicit realization of those ideas. 
We consider the sheaf ${\mathfrak F}$ of global sections of $G$-valued formal series associated to an 
infinite-dimensional Lie algebra 
${\mathfrak g}$, and the sheaf $\mathcal G({\bf z})$ of $G$-valued mero functions. 
According to \cite{HinSch}, for any sheaf of Lie algebras $\mathfrak h$ there is another sheaf of differential
graded Lie algebras constituting a resolution of $\mathfrak h$.
 It is the sheaf of cosimplicial Lie
algebras given by taking ${\mathfrak h}$ on the ${\rm {\check C}}$ech complex associated to a covering $\U$
by Stein
open sets, suitably normalized by the Thom-Sullivan functor \cite{HinSch}. 
For ${\mathfrak F}$ the graded differential algebra $\Gamma(\mathfrak F, M)$ of global sections of ${\mathfrak F}$  
is obtained by application of the Chevalley-Eilenberg 
complex 
 on a Stein cover of $M$.
On the other hand the graded differential algebra $\check C(G, \U)$
is given through Chevalley-Eilenberg simplicial chain complex 
 on a Stein open sets
with multiplication \eqref{multiplication}. 
Our aim now is to construct an explicit isomorphisms of these two complexes. 

We then find (following the lines of \cite{HinSch, wag})  
a relation between cohomology of   
the sheaf of graded differential algebras associated to $\mathfrak g$   
and 
 the sheaf of graded differential algebra ${\check C}(G, \U)$ of mero functions for a Lie algebra $\mathfrak g$.  
Let, as in \cite{wag}, denote by  
 $N_{*}$ the thickened nerve of the covering
${\U}$, i.e., the simplicial complex manifold associated to the
covering ${\U}$.
On $M$, there exists an inclusion
\begin{equation}
\label{inclusion}
f : 
 \check C(G, N_{M,q}) \hookrightarrow {\mathfrak F}(G, N_{M,q}), 
\end{equation}
of graded differential algebras $\check C(G, N_{M,q})$ and ${\mathfrak F}(G, N_{M,q})$ on $N_{M,q}$. 
By applying the modification of the Quillen functor \cite{wag}, 
this inclusion induces 
\begin{displaymath}
\widetilde{f} : C^*_{dg}({\mathfrak F}, N_{*}) \to C^*_{cont}(\mathcal G({\bf z}),  N_{*}),  
\end{displaymath}
a morphism of simplicial cochain complexes.  
 By Proposition 5.9 in \cite{BS}, the morphism $\widetilde{f}$
induces a cohomology equivalence between the realizations of the two simplicial
cochain complexes $C_k^l(\mathcal G(\bf z), M)$ and $C^l_k(G, \U)$.  
 The conditions of the lemma are fulfilled because
of the isomorphism of the cohomologies on a Stein open set of
the covering and the K\"unneth theorem \cite{wag}.  
Using  Proposition 
6.2 of 
\cite{BS}, and involving partitions of unity, one shows that the cohomology
of the realization of the simplicial cochain complex on the left hand
side gives the graded differential cohomology of $\Ga(X,\mathfrak{g})$.
\end{proof}
\section{Relation of cosimplicial and singular cohomology}
%
Gelfand and Fuks \cite{Fuk}  
calculated cohomology of the Lie algebra of formal vector fields in $n$ complex variables 
$W_n$.   
In particular, they proved \cite{Fuk, FF1988}
\begin{theorem} 
There exists a manifold $X(n)$ such that the continuous cohomology of $W_n$ is equivalent to singular 
cohomology of $X(n)$ 
\begin{displaymath}
H^*_{cont}(W_n) \cong H^*_{sing}(X(n)).
\end{displaymath}
\end{theorem}
In \cite{BS} they  
showed that for $\R^n$ or, more generally, for a starshaped open set $U$ of an 
$n$-dimensional manifold $M$, the Lie algebra of
$C^{\infty}$-vector fields $Vect(U)$ has the same cohomology as $W_n$.
%
In \cite{wag} it was proven that 
 the same is true for the Lie algebra of holomorphic vector fields on a
disk of radius $R$ in $\C^n$.  
In this paper, we consider cohomology of mero 
functions provided by bilinear pairings for an 
 arbitrary $n$-formal parameter Lie algebra $G$-valued series localized on a 
complex $n$-dimensional manifold $M$. 
For purposes of determining the cosimplicial cohomology, we use the machinery of 
chequered necklaces \cite{MT} associated with $G$-valued series. 
\subsection{A-matrix}
In this subsection we
discuss a number of elliptic 
functions that we
will need. 
The Weierstrass elliptic function with periods
$\sigma$, $\varsigma \in \mathbb{C}^{\ast }$ is defined by 
\begin{equation}
\wp (z,\sigma ,\varsigma )=z^{-2}+ \sum_{m,n\in \mathbb{Z},(m,n)\neq
(0,0)} \left[ (z-m\sigma -n\varsigma )^{-2}- (m\sigma +n\varsigma
)^{-2}\right].  \label{Weierstrass}
\end{equation}%
Choosing $\varsigma =2\pi i$ and $\sigma =2\pi i\tau $ (where $\tau$ 
 lies in the complex upper half-plane $\mathbb{H}$), we define 
\begin{eqnarray}
P_{2}(\tau ,z) &=&\wp (z,2\pi i\tau ,2\pi i)+E_{2}(\tau )  \notag \\
&=&\frac{1}{z^{2}}+\sum_{k=2}^{\infty }(k-1)E_{k}(\tau )z^{k-2}.  \label{P2}
\end{eqnarray}%
Here, $E_{k}(\tau )$ is equal to $0$ for $k$ odd, and for $k$ even is the
Eisenstein series \cite{Se} 
\begin{equation*}
E_{k}(\tau )=-\frac{B_{k}}{k!}+\frac{2}{(k-1)!}\sum_{n\geq 1}\sigma
_{k-1}(n)q^{n}.
\end{equation*}%
Here we take $q=\exp (2\pi i\tau )$; $\sigma
_{k-1}(n)=\sum_{d\mid n}d^{k-1}$, and $B_{k}$ is the $k$th Bernoulli number
defined by 
\begin{eqnarray*}
\frac{t}{e^{t}-1}-1+\frac{t}{2} 
=
\sum_{k\geq 2}B_{k}\frac{t^{k}}{k!} 
=
{\frac{1}{12}}{t}^{2}-{\frac{1}{720}}{t}^{4}+{\frac{1}{30240}}{t}^{6}+O({t%
}^{8}).
\end{eqnarray*}%
$P_{2}$ can be alternatively expressed as 
\begin{equation}
P_{2}(\tau ,z)=\frac{q_{z}}{(q_{z}-1)^{2}}+\sum_{n\geq 1}\frac{nq^{n}}{%
1-q^{n}}(q_{z}^{n}+q_{z}^{-n}),  \label{P2exp}
\end{equation}%
where $q_{z}=\exp (z)$.
 If $k\geq 4$ the
 first three Eisenstein series $E_{2}(\tau ),E_{4}(\tau ),E_{6}(\tau )$
are algebraically independent and generate a weighted polynomial algebra 
$Q=
\mathbb{C}[E_{2}(\tau )$, $E_{4}(\tau ),E_{6}(\tau )]$.
We define $P_{1}(\tau ,z)$ by 
\begin{equation}
P_{1}(\tau ,z)=\frac{1}{z}-\sum_{k\geq 2}E_{k}(\tau )z^{k-1},  \label{P1}
\end{equation}%
where $P_{2}=-\frac{d}{dz}P_{1}$ and $P_{1}+zE_{2}$ is the classical
Weierstrass zeta function. 
We also define $P_{0}(\tau ,z)$, up to a choice of the logarithmic branch,
by 
\begin{equation}
P_{0}(\tau ,z)=-\log (z)+\sum_{k\geq 2}\frac{1}{k}E_{k}(\tau )z^{k},
\label{P0}
\end{equation}%
where $P_{1}=-\frac{d}{dz}P_{0}$. 
Define elliptic functions $P_{k}(\tau ,z)\,$ for $k\geq 3\,$ from the
analytic expansion 
\begin{equation}
P_{1}(\tau ,z-w)=\sum_{k\geq 1}P_{k}(\tau ,z)w^{k-1}  \label{P1Pnexpansion}
\end{equation}
where 
\begin{equation}
P_{k}(\tau ,z)=\frac{(-1)^{k-1}}{(k-1)!}\frac{d^{k-1}}{dz^{k-1}}P_{1}(\tau
,z)=\frac{1}{z^{k}}+E_{k}+O(z).  \label{Pkdef}
\end{equation}
Finally, define for $k,l=1,2,\ldots $ 
\begin{eqnarray}
C(k,l) &=&C(k,l,\tau )=(-1)^{k+1}\frac{(k+l-1)!}{(k-1)!(l-1)!}E_{k+l}(\tau ),
\label{Ckldef} \\
D(k,l,z) &=&D(k,l,\tau ,z)=(-1)^{k+1}\frac{(k+l-1)!}{(k-1)!(l-1)!}%
P_{k+l}(\tau ,z).  \label{Dkldef}
\end{eqnarray}%
$\,$Note that $C(k,l)=C(l,k)$ and $D(k,l,z)=(-1)^{k+l}D(l,k,z)$. These
naturally arise in the analytic expansions (in appropriate domains) 
\begin{equation}
P_{2}(\tau ,z-w)=\frac{1}{(z-w)^{2}}+\sum_{k,l\geq 1}C(k,l)z^{l-1}w^{k-1},
\label{P2expansion}
\end{equation}%
and for $k\geq 1$ 
\begin{eqnarray}
P_{k+1}(\tau ,z) &=&\frac{1}{z^{k+1}}+\frac{1}{k}\sum_{l\geq 1}C(k,l)z^{l-1},
\label{Pkexpansion} \\
P_{k+1}(\tau ,z-w) &=&\frac{1}{k}\sum_{l\geq 1}D(k,l,w)z^{l-1}.
\label{Pkzwexpansion}
\end{eqnarray}
Notation here is as follows: $A(\tau,\epsilon )$ is the infinite 
matrix with $(k,l)$-entry 
\begin{equation}
A(k,l,\tau,\epsilon )=\frac{\epsilon ^{(k+l)/2}}{\sqrt{kl}} 
C(k,l,\tau);  \label{Aki1}
\end{equation}%
Our setup is 
facilitated by an alternate description in terms of combinatorial
gadgets that we call chequered necklaces \cite{MT}. 
 They are certain kinds of
graphs with nodes labeled by positive integers and edges labeled by
quasimodular forms, and they play an important role in this section. 
It is useful to introduce an interpretation for complexes  
  in terms of the sum of weights of certain graphs. 
In particular, in this subsection, we construct a special simplicial manifold
defined
over the space of chequered necklaces. 
\subsection{Explicit construction of the simplicial manifold $X(G,n)$} 
Now we will show how to associate elements of $G$ to chequered necklaces. 
Recall that in general each element of $G$ has a form $W_g(z_i)=\sum\limits_{m \in \Z} g_{i, m} z^m$, 
Consider the elements 
$F(x_1, x_2, (h, \zeta))= ( \theta, 
W_{\overline{h}}(\zeta)\; W_{g_1}(z_1) \; W_{g_2}(z_2)   \; W_{h}(\zeta) )$, that 
belong to the space  
 $\Theta(ln, kn)$.  
Here $\overline{h}$ is dual to $h \in G$ with respect to the bilinear pairing. 
 Performing the summation over $h\in G$, 
$G(x_1, x_2, \zeta)=\sum\limits_{h \in G} (\theta, 
W_{\overline{h}}(\zeta)\; W_{g_1}(z_1) \; W_{g_2}(z_2)   \; W_{h}(\zeta) )$, 
we obtain (with appropriate choice of the complex parameter $\zeta$ \cite{MT}),  
an element $C(g_1, g_2, z_1-z_2, \tau)$ which expandes as \eqref{Aki1} in terms of $C(m, m', \tau)$ 
or $D(m, m', z_1-z_2, \tau)$. 
Thus we see make a connection between $G$-elements and a chequered necklaces.  

Let us explicitly construct the special manifold $X(G, n)$.  
A simplicial manifold is a simplicial complex for which the geometric realization 
is homeomorphic to a topological manifold.  
This is essentially the concept of a triangulation in topology.  
This can mean simply that a neighborhood of each vertex 
(i.e., the set of simplices that contain that point as a vertex) is homeomorphic to an $n$-dimensional ball.
%
Now let us see how sets of chequerd necklaces turn into a simplicial manifold. 
Chequered necklaces are one-simplexes in form of graphs. 
They are associated to elements of $G$. 
If ${\bf z}_n$-dependence of $G$-elements is taken into account then 
a chequered necklaces is considered as an $(n+2)$-simplex (due to dependence on extra two complex parameters
$\tau$ and $\epsilon$.  
The set of chequered necklaces $\mathcal{N}_m$, $m \ge 2$, is a set of special graphs. 
The spaces of simplicial complex are sums of weights of certain graphs (period matrix). 
Chequered necklaces $\mathcal{N}_m$ are connected graphs with $m \geq 2$ nodes, $(m-2)$ of 
which have valency $2$ and two of which have valency $1$ (these latter are the end nodes), 
together with an orientation, 
on the edges. 
These graphs represent simpleces. 
Graphs have vertices
 labeled by positive integers and edges
are labeled alternatively by $1$ or $2$ as one moves along the graph, e.g., 
\begin{equation*}
\ \overset{k_{1}}{\bullet }\overset{1}{\longrightarrow }\overset{k_{2}}{%
\bullet }\overset{2}{\longrightarrow }\overset{k_{3}}{\bullet }\overset{1}{%
\longrightarrow }\overset{k_{4}}{\bullet }\overset{2}{\longrightarrow }%
\overset{k_{5}}{\bullet }\overset{1}{\longrightarrow }\overset{k_{6}}{%
\bullet }
\end{equation*}
\end{definition}
To complete definition of neighborhoods, one defines a weight function
\begin{equation*}
\omega :\mathcal{N}_m \longrightarrow 
\mathbb{C}[E_{2}(\tau), E_{4}(\tau), E_{6} (\tau),\epsilon],  
\end{equation*}
as follows: if a chequered necklace $\mathcal N_n$ has edges $E$ labeled as 
$\overset {k}{\bullet }\overset{a}{\longrightarrow }\overset{l}{\bullet }$, $a=1$, $2$ then we 
define 
\begin{eqnarray}
\omega (E) &=&A(k,l,\tau, \epsilon ),  
\notag
 \\
\omega (\mathcal N_m) &=&\prod
\omega (E),  \label{wtedge}
\end{eqnarray}
where $A(k,l,\tau,\epsilon)$ is given by 
(\ref{Aki1}) and the
product is taken over all edges $E$ of $\mathcal N_m$. 
Here $\tau$, $\epsilon$ are complex parameters. 
We further define 
$\omega(\mathcal N_{0})=1$.
Recall that $\left\{E_{2}(\tau), E_{4}(\tau), E_{6} (\tau)\right\}$
form a basis for the Eisenstein series. 
After a normalization 
the weight functions define a map from $\mathcal{N}_m$ to an $(n+2)$-dimensional ball in $\C$.  
The metrix on $X(G, n)$ is given by non-degenerate bilinear pairing $(., .)$. 
Since $\omega(\mathcal N)$ is a map from $\mathcal N$ to 
$\mathbb{C}[E_{2}(\tau_{a}), E_{4}(\tau_{a}), E_{6} (\tau), \epsilon]$, 
 all the transition function of $X$ can be expressed via $\omega(\mathcal N)$. 
Together with functions on graphs these simpleces are homeomorphic to an $n$-dimensional ball. 
%

Now let us associate chequered necklace to elements of $G$ and form $X(G, n)$.  
According to the construction of Section \ref{functional}, 
 non-commutative coefficients of a formal Lie-algebraic series are elements of the algebraic completion $G$ of a
$\mathcal G$-module. 
For each element $g_{j, k}$, $k \in \C$, $j \ge \Z$, associate a diagram \cite{MT} representing it as action of 
generators on the union element. 
The properties of non-degenerate bilinear pairing $(.,.)$ allow us to find appropriate diagram for 
the element $\widetilde g_{j, k}$ of $\widetilde G$ dual to $g_{j, k}$. 
Recall the chequered necklace construction for elements $g_{j, k}$ used in \cite{MT}. 
 Chequered necklace is in one to one correspondence with the formation of an $G$-element, 
Associate a knot of such diagram to a point of $X(G, n)$. 
Each point of the $G$-necklace is endowed with a power of $z_j$, $j \in \Z$. 
Let us associate to a point on the $G$-necklace the 
zero power of $z_j$ the zero point of a local domain $U_j$ ob $X(G, n)$. 
The union $V(G, n)$ of all chequered necklaces together with local domains $U_j$ present 
in the definition of the double complex constitutes the cells of a skeleton for $G$. 

Now let us define the cohomology of the simplicial manifold $X(G, n)$. 
In our setup, the 
  cohomology is an invariant 
associating a graded ring with the manifold $X(G, n)$. 
Every continuous map $f: X \to Y$ determines a homomorphism from the cohomology 
ring of Y to that of X; this puts strong restrictions on the possible maps from X to Y. 
Unlike more subtle invariants such as homotopy groups, the cohomology ring tends to be 
computable in practice for spaces of interest.
Now we are able to define the chain-cochain complex as the space 
 $C^m_k(\omega_k(\mathcal N_m))= \left\{ \omega_k(\mathcal N_m)\right\}$, 
 of $\omega$-forms 
associated to a vertex on $X(G, n)$ as simplicial complex. 
The boundary operator defined as
$d_k(\omega_k(\mathcal N_{m+1}) ) = P_{\mathcal N}.\omega_{k}(\mathcal N_m)$,  
where the operator $P_{\mathcal N}$ extends an $m$-vertex chequered necklace by one vertex. 
Note that according to properties of meromorphic functions with prescribed behavior, 
and the identification of $G$-elements with necklace elements, 
the extension of a chequered necklace leads to the shift $k\to (k-1)$. 
The cohomology of $X(G, n)$ is defined as the cohomology of the complex $C^m_k(\omega(\mathcal N_m))$.  
\subsection{A counterpart of Bott-Segal theorem}
In this subsection we prove an analogue of Bott-Segal theorem \cite{BS}. 
Being equipped with the technique of chequered necklaces, we prove the following
\begin{proposition} 
\label{guga}
There exists a manifold $X(G, n)$ such that the cosimplicial cohomology of mero 
functions for $\mathfrak g$    
 on a smooth complex manifold $M$   
is equivalent to the  
cohomology of $X(G, n)$, i.e.,  
\begin{displaymath}
\label{quota}
H^*_{* \; cos}(G, \; \U    
) \cong H^*_{*}(X(G, n)). 
\end{displaymath}
\end{proposition}
%
\begin{proof}
The manifold $X(G, n)$ and 
 its cohomology $H^*_{* \; sing}(X(G, n))$ were constructed in previous subsections.   
According to the construction of Section \ref{functional}, 
 non-commutative coefficients of a formal Lie-algebraic series are elements of a $\mathcal G$-module $G$. 
For each element $g_{j, k}$, $k \in \C$, $j \ge \Z$, we associate (by using the summation over matrix elements 
described in corresponding subsection) 
 a diagram \cite{MT} representing it as action of 
generators on the union element. 
The properties of non-degenerate bilinear pairing $(.,.)$ allow us to find appropriate diagram for 
the element $\widetilde g_{j, k}$ of $\widetilde G$ dual to $g_{j, k}$. 
Recall the chequered necklace construction for elements $g_{j, k}$ used in \cite{MT}. 
Associate a knot of such diagram to a point of $X(G, n)$. 
Each point of the $G$-necklace is endowed with a power of $z_j$, $j \in \Z$. 
Let us associate to a point on the $G$-necklace the 
zero power of $z_j$ the zero point of a local domain $U_j$ ob $X(G, n)$. 
The union $V(G, n)$ of all chequered necklaces together with local domains $U_j$ present 
in the definition of the double complex constitutes the cells of a skeleton for $G$. 
Thus, we obtain an analog of a $2n$-skeleton for $W_n$ of formal vector fields.
In contrast to \cite{Fuk, FF1988} it is endowed with a power of $j$-th formal parameter $z_j$. 
We define a map  
$\pi: V(G,n) \to G(\infty,n)$,    
from the the $G$-skeleton  
to an infinite Grassmanian $G(\infty, n)$. 
Since the inverse image of the union of the cells is not a manifold, 
we consider 
 an open neighborhood of the inverse image under $\pi$ of the $G$-skeleton of the Grassmannian
$G(\infty, n)$. 
The union of such open neighborhoods constitutes the manifold $X(G, p)$. 
This is the cohomology of double Lie-algebraic complexes $C^l_k(\Theta(n,k) )$ which is the union of complexes 
 $C^l_k(\Theta(1, k) )$ for each local coordinate and $g_{jk}$-generators.  
It coincides with 
 the cosimplicial cohomology $H^*_{*\; cos}( {\mathcal G}({\bf z}), \U)$ of $W_n$ of $n$ complex 
variables. 
\end{proof}
\begin{remark}
It is also analogue of cohomology equivalence of $H^*_*(G, \U)$ and cohomology of a groupoid 
shown in \cite{CM}.  
\end{remark}
Note that another way to prove Proposition \ref{guga} is to use the same technique as in
\cite{Fuk, FF1988} since 
 for the complex $C^l_k({\mathcal G}({\bf z}), \U)$ there exist converging spectral sequences.   
i.e., to show that 
an isomorphism
of the Hochschild-Serre spectral sequence \cite{Ho} for the subalgebra $gl(n)$
with the Leray spectral of the restriction to the $2n$ skeleton of
the universal $U(n)$ principal bundle. 
\section{Conclusions}
In this section we list multiple  
 applications of the research of this paper are in conformal field theory 
\cite{Fei, wag, BZF, BPZ, DiMaSe, TUY}, in deformation theory \cite{BG, HinSch}, and in the theory
of foliations \cite{Bott}. 
\subsection{Applications in conformal field theory and moduli spaces}
In \cite{Fei, BeiFeiMaz} 
applications of cosimplicial computations on compact Riemann surfaces in 
conformal field \cite{DiMaSe, BZF, BPZ} theory were treated. 
As we deal with special homology, we replace the sheaf of
holomorphic vector fields 
by the sheaf of meromorhic functions associated to corresponding Lie algebra.   
In \cite{Fei}, for Riemann surface $\Sigma^{(g)}$, 
Feigin calculated the cosimplicial homology of
$Lie(\Sigma^{(g)})$ with values in the 
representations mentioned in Introduction. 
It is possible to compute cosimplicial homology of a space of meromorphic function complexes 
associated to various Lie algebras.  
The space of coinvariants on the right hand side
 defining so-called modular functor is usually associated to
locally defined objects. 
We will 
 obtain 
its homological description in terms of 
globally defined objects.
The space of coinvariants supposed to be 
the continuous
dual to the local ring completion  of the moduli space of
compact Riemann surfaces of genus $g\geq 2$ at the point $\Sigma^{(g)}$,
provided that $\Sigma^{(g)}$ is a smooth point. 
This gives an important link between
Lie algebra homology and the geometry of the moduli space. 

\subsection{Applications in deformation theory}
%
{\it Deformations of complex manifolds.} 
Cosimplician considerations above are applicable to
  cohomology computations 
in the 
 deformation theory of complex manifolds 
\cite{Ma, Fei, HinSch, GerSch}. 
%
The completion of a local ring of moduli space at a given point $M$
is isomorphic to the dual of the Lie algebra of $M$-infinitesimal automorphisms zero-th homology group.  
 This links Lie algebra homology and
  geometry of the moduli space in a formal neighborhood of a point.
We expect results in this direction for higher dimensional complex manifolds. 
In \cite{wag} we find the condition for the first cohomology in the case of 
higher dimensional
complex manifolds $M$. 
For restricted function cohomology one can consider also related the deformation 
theory following Kodaira and Spencer \cite{Kod}.  

\medskip 
{\it Deformations of Lie algebras.}
 It is well known that the Lie algebra cohomology 
with values in the adjoint representation $H^*(L,L)$ of a Lie algebra
$L$ answers questions about deformations of $L$ as an algebraic
object. 
For example, $H^2(L,L)$ can be interpreted as the space of
equivalence classes of infinitesimal deformations of $L$, see
\cite{Fuk, FF1988}. 
There arise natural questions of this type for  
 bi-graded differential  Lie algebras resulting from chain complex constructions.   
For a disk 
$D\subset\C^n$, holomorphic vector fields   
are rigid, i.e., 
\cite{wag} 
$H^*_{cont}(Hol(D),Hol(D)) = 0$. 
Using cosimplicial cohomology results, we will study rigidity of  
  bi-graded differential Lie algebras resulting from chain complex constructions. 
For a 
compact Riemann surface $\Sigma^{(g)}$
of genus $g\geq 2$, we expect to find a relations for cohomologies in terms of 
elements of Fr\'echet spaces 
 given by the polynomials on $T_{\Sigma^{(g)}}{\mathcal N}(g,0)$.
  It's the space of formal power series on $T_{\Sigma^{(g)}}{\mathcal N}(g,0)^*$. 
This could be interpreted as a relation between cohomology with
adjoint coefficients of $\mathfrak{g}$, 
i.e., graded differential 
deformations of global sections of $\mathfrak{g}$, and deformations
of the underlying manifold.
As it is explained in \cite{Fei, wag},  
the choice of the coefficients in the Lie algebra cohomology 
determines a geometric object on the moduli space in a formal
neighborhood of a point. 
Namely,  trivial coefficients correspond to the
structure sheaf, adjoint coefficients correspond to vector fields,
adjoint coefficients in the universal enveloping algebra correspond to
differential operators.     

\subsection{Applications in foliation theory.}
Applications in foliation theory are inspired by the link between cohomology of
Lie algebras and characteristic classes of foliations 
\cite{Fuk, FF1988}. 
In \cite{wag} the author considered the  
case of characteristic classes of $g$-structures. 
%
For a complex manifold $M$, and ${\mathcal U} = \{U_i\}_{i\in I}$ a covering of
$M$ by open sets such that $I$ is a countable directed index
set, 
consider the sheave   
%
of meromorphic functions in cosimplicial setup given above. 
Denote by $W_{2n}|_{hol}$ the Lie subalgebra of
$W_{2n}$ generated by the $\frac{\partial}{\partial z_i}$ for
$i=1,\ldots,n$.
  Given a 
structure of meromorphic functions associated to such a covering,  
we obtain 
 that 
the space of moduli of
$\{\omega_U\}_{U\in{\mathcal U}}$ is isomorphic 
to the moduli space of 
$W_{2n}|_{hol}$-valued differential forms $\omega$.
 To such a structure 
we may assign as in \cite{wag} 
  characteristic classes by considering   
$H^*(|C^*_{cont}( {C}({\mathcal U}, \F 
))|)$. 
The 
cosimplicial meromorphic function 
 structure is defined 
 such that by inserting $p$-times
$\chi 
_{U_{i_0}\cap\ldots\cap U_{i_q}}$ into each 
$c\in C^p_{cont}(\prod_{i_0<\ldots<i_q} \F 
(U_{i_0}\cap\ldots\cap U_{i_q}))$, 
one associates 
an element $\chi$   
of the generalized ${\rm   C}$ech-de Rham complex 
associated to the covering ${\mathcal U}$ on $M$.
 By the standard reasoning 
this $\chi$ 
provide 
a well-defined cohomology class $[ \chi 
]$,
the characteristic class associated to the cosimplicial 
$\F$-structure.  

Finally, a relation to factorization algebras \cite{HK} will be cleared elsewhere.  
\section*{Acknowledgments}
The author would like to thank H. V. L\^e, A. Lytchak, P. Somberg, and P. Zusmanovich for related discussions. 

\end{document}